\documentclass[11pt]{amsart}

\usepackage{amsmath, amsthm, amssymb} 

\usepackage{hyperref}

\usepackage{graphicx}

\newtheorem{theorem}{Theorem}[section]
\newtheorem{proposition}[theorem]{Proposition}
\newtheorem{lemma}[theorem]{Lemma}

\newtheorem{corollary}[theorem]{Corollary}

\newtheorem*{claim}{Claim}

\theoremstyle{definition}
\newtheorem{definition}[theorem]{Definition}

\theoremstyle{remark} 
\newtheorem{remark}[theorem]{Remark}

\title[Van Lambalgen's Theorem for uniformly relative randomness]{Van Lambalgen's Theorem for uniformly relative Schnorr and computable randomness}
\author{Kenshi Miyabe and Jason Rute}

\newcommand{\citelow}[2][def]{\cite[#1]{#2}} 

\newcommand{\df}[1]{\emph{\textbf{#1}}}

\newcommand{\altmid}{\,:\,}

\begin{document}

\begin{abstract}
We correct Miyabe's proof of van Lambalgen's theorem for truth-table Schnorr randomness (which we will call uniformly relative Schnorr randomness).  An immediate corollary is one direction of van Lambalgen's theorem for Schnorr randomness.  It has been claimed in the literature that this corollary (and the analogous result for computable randomness) is a ``straightforward modification of the proof of van Lambalgen's theorem."  This is not so, and we point out why.
We also point out an error in Miyabe's proof of van Lambalgen's theorem for truth-table reducible randomness (which we will call uniformly relative computable randomness).  While we do not fix the error, we do prove a weaker version of van Lambalgen's theorem where each half is computably random uniformly relative to the other.
We also argue that uniform relativization is the correct relativization for all randomness notions.
\end{abstract}

\maketitle

\section{Introduction}
\label{sec:Intro}

Recall van Lambalgen's theorem.

\begin{theorem}[van Lambalgen \citelow{Lambalgen:1990yq}]\label{thm:VL}
$A \oplus B$ is Martin-L{\"o}f random if and only if $A$ is Martin-L{\"o}f  random and $B$ is Martin-L{\"o}f  random relative to $A$.
\end{theorem}

Merkle et al.\ \citelow{Merkle:2006vn} showed that the ``$\Rightarrow$" direction of van Lambalgen's theorem does not hold for Schnorr or computable randomness.  This has been extended by Yu \citelow{Yu:2007rt}, Kjos Hanssen \citelow[Remark~3.5.22]{Nies:2009qf}, Franklin and Stephan \citelow{Franklin:2011fr}, and Miyabe \citelow{Miyabe:2011zr}.

In \citelow{Miyabe:2011zr} the first author claimed that van Lambalgen's theorem does in fact hold for Schnorr randomness if the usual notion of relativized Schnorr randomness is replaced with the weaker notion of uniformly relative Schnorr randomness---previously called truth-table Schnorr randomness in \citelow{Franklin:2010ys} and \citelow{Miyabe:2011zr}.

\begin{theorem}\label{thm:VL-tt-Schnorr-intro}
$A \oplus B$ is Schnorr random if and only if $A$ is Schnorr random and $B$ is Schnorr random uniformly relative to $A$.
\end{theorem}

However, the proof given was incorrect and we provide a corrected proof.  Our proof follows a standard proof of van Lambalgen's theorem using integral tests, except at the difficult point we apply a key lemma, which can be seen as an effective version of Lusin's theorem for a particular setting.  Lusin's theorem, one of Littlewood's three basic principles of measure theory, is the basis behind the layerwise-computability framework that has been successively employed by Hoyrup, Rojas and others to relate algorithmic randomness and computable analysis.

The structure of the paper is as follows.  In Section~\ref{sec:key-lemma}, we prove the key lemma and some corollaries.

In Section~\ref{sec:vL-for-Schnorr}, as a warm-up, we show how our key lemma can be used to prove the ``$\Leftarrow$" direction of van Lambalgen's theorem for Schnorr randomness.  (A different proof was given recently by Franklin and Stephan \citelow{Franklin:2011fr}.)  Yu \citelow{Yu:2007rt} had claimed  that ``the [$\Leftarrow$] direction of van Lambalgen's theorem is true for both Schnorr randomness and computable randomness. [...] The proof is just a straightforward modification of the proof of van Lambalgen's theorem."  Downey and Hirschfelt \citelow{Downey:2010ve} had made similar claims.  Unfortunately, the proofs are not so straightforward, and we explain why in the case of Schnorr randomness.

In Section~\ref{sec:tt-Schnorr}, we define uniformly relative Schnorr randomness, and prove van Lambalgen's theorem for this notion of randomness.

In Section~\ref{sec:comp-rand}, we discuss uniformly relative computable randomness---previously called truth-table reducible randomness.  We remark that Miyabe's \citelow{Miyabe:2011zr} proof of van Lambalgen's theorem for uniformly relative computable randomness is not correct in the ``$\Leftarrow$" direction.  While we do not provide a correction, we do prove the following weaker result.

\begin{theorem}
$A \oplus B$ is computably random if and only if each of $A$ and $B$ are computably random uniformly relative to the other.
\end{theorem}

This weakening of van Lambalgen's theorem is also known to hold for Kolmogorov-Loveland randomness. We leave as an open question the ``$\Leftarrow$" direction of van Lambalgen's theorem for both computable randomness and uniformly relative computable randomness.  We conjecture that it is false for both.

Finally, in Section~\ref{sec:equiv} we prove that Franklin and Stephan's \citelow{Franklin:2010ys} truth-table Schnorr randomness is equivalent to our uniformly relative Schnorr randomness.  The difference in terminology reflects the difference in definitions, and we discuss why the Franklin and Stephan definition, which uses truth-table reducibility, is very sensitive to the choice of test used.

We believe this paper gives a strong argument that uniform relativization (as in Definition~\ref{def:Uniform-Schnorr}) is the correct method to relativize a randomness notion (in contrast to the usual method of relativization).  It is a natural definition that can be applied to all the standard randomness notions.  Moreover, uniformly relative Martin-L{\"o}f randomness is equivalent to the usual relative Martin-L{\"o}f randomness (Section~\ref{sub:rmk-on-KL-ML}). Uniformly relative Schnorr randomness not only satisfies van Lambalgen's theorem, but also has well-behaved lowness properties \citelow{Franklin:2010ys}.  Furthermore, it is not difficult to see that uniformly relative Demuth randomness is equivalent to Demuth$_\text{BLR}$ randomness (see \citelow{Bienvenu:2013cr}, \citelow{Diamondstone:2013dq}), which also satisfies van Lambalgen's theorem \citelow{Diamondstone:2013dq} and has natural lowness properties \citelow{Bienvenu:2013cr}.  Indeed we suggest that if one wishes to explore either van Lambalgen's theorem or low-for-randomness with respect to other randomness notions, one should use uniform relativization.

\section{The key lemma}
\label{sec:key-lemma}

For this paper, we will work in $2^\omega$ with the fair-coin measure $\mu$.  Recall the following definitions.  The reader may wish to consult the books \citelow{Downey:2010ve, Nies:2009qf} for further background.

\begin{definition}\label{def:open}
A set $U \subseteq 2^\omega$ is \df{open} if it is a countable union of \df{basic open sets}, i.e.\ sets of the form $[\sigma]:=\{X\in 2^\omega \altmid X \succ \sigma\}$ for some $\sigma \in 2^{<\omega}$ as well as the empty subset $\varnothing \subset 2^\omega$.  A \df{code} for an open set $U\subseteq 2^\omega$ is a sequence $\langle C_s \rangle_{s\in \mathbb{N}}$ of basic open sets such that $U = \bigcup_s C_s$.  A set $U \subseteq 2^\omega$ is $\Sigma^0_1$, or \df{effectively open}, if it is open with a computable code.
\end{definition}

\begin{definition}
A \df{Martin-L{\"o}f test} is a uniform sequence $\langle U_n\rangle$ of $\Sigma^0_1$ subsets of $2^\omega$ such that $\mu(U_n) \leq 2^{-n}$.  A \df{Schnorr test} is a Martin-L{\"o}f test $\langle U_n\rangle$ such that $\mu(U_n)$ is uniformly computable in $n$.  A set $X \in 2^\omega$ is said to be \df{covered by} a Martin-L{\"o}f (Schnorr) test $\langle U_n\rangle$ if $X\in \bigcap_n U_n$.  The set $X \in 2^\omega$ is said to be \df{Martin-L{\"o}f} (resp.\ \df{Schnorr}) \df{random} if $X$ is not covered by any Martin-L{\"o}f (resp.\ Schnorr) test.
\end{definition}

\begin{definition}
A function $f\colon 2^\omega \rightarrow [0,\infty]$ is \df{lower semicomputable} if there is a uniform sequence of total computable functions $g_n \colon 2^\omega \rightarrow [0,\infty)$ such that $f = \sum_n g_n$.
\end{definition}

A more standard definition of lower semicomputable is that $f$ is lower semicomputable if $f(X)$ is uniformly lower semicomputable (left c.e.)\ from $X$. Our definition is easily seen to be equivalent. (This even remains true when $2^\omega$ is replaced with the unit interval or another computable Polish space.)

Recall the following definitions.
\begin{definition}
A function $f\colon 2^\omega \rightarrow \mathbb{R}$ is \df{$L^1$-computable} if there is a uniformly computable sequence of bounded computable functions $\langle g_n \rangle$ such that  $\| f-g_n \|_{L^1} \leq 2^{-n}$. (Since $2^\omega$ is compact, all computable functions are bounded.)

The \df{distribution} of a function $f\colon 2^\omega \rightarrow \mathbb{R}$ is the probability measure $\nu$ on $\mathbb{R}$ defined by  $\nu(A) = \mu(\{X\in 2^\omega \altmid f(X) \in A\})$ for all Borel sets $A\subseteq\mathbb{R}$.  This is also known as the \df{pushforward  of the fair-coin measure along $f$}.

A distribution $\nu$ is \df{computable} if $\nu(U)$ is lower semicomputable (left c.e.)\ uniformly from any code for an open set $U\subseteq\mathbb{R}$.\footnote{%
This definition is equivalent to $\nu$ being a computable point in the L{\'e}vy-Prokhorov metric \citelow{Hoyrup:2009bh}. It is also equivalent to the map $f\mapsto\int\! f\, d\nu$ being a computable operator on bounded continuous functions $f$ \citelow{Schroder:2007kx}.%
}  
(A code for an open set $U\subseteq\mathbb{R}$ is the same as in Definition~\ref{def:open}, except that the basic open sets are open intervals with rational endpoints.)
\end{definition}

We could not find a direct proof of this next fact, so we give one here.

\begin{proposition}\label{prop:dist-comp}
Let $f\colon 2^\omega \rightarrow \mathbb{R}$ be an $L^1$-computable function with distribution $\nu$.  Then $\nu$ is a computable distribution.
\end{proposition}

\begin{proof}It is enough to prove that
\[\nu((a-r,a+r))=\mu (\{X\in 2^\omega \altmid |f(X) - a| <  r \})\]
is lower semicomputable from $a,r$.  (An open set $U\subseteq\mathbb{R}$ is encoded as a union of countably many rational intervals. To lower semicompute $\nu(U)$ it is enough to lower semicompute the measure of each finite subunion of intervals.  A finite union of intervals can be made a disjoint union by joining overlapping intervals.)

If $f$ is computable, then we are done, since $\{X\in 2^\omega \altmid |f(X) - a| <  r \}$ is $\Sigma^0_1$ relative to $a$ and $r$.

Otherwise, we know
\begin{align}\label{eq:sup}
\begin{split}
&\mu (\{X\in 2^\omega \altmid |f(X) - a| <  r \}) \\
&\qquad \qquad= \sup_{\varepsilon>0} \mu(\{X\in 2^\omega \altmid |f(X) - a| <  r - \varepsilon \}).
\end{split}
\end{align}
For any $\varepsilon>0$ and $\delta>0$ we can effectively approximate $f$ by some computable function $g$ such that $\Vert f-g \Vert_{L^1} < \varepsilon \cdot \delta$.  Then by Chebeshev's inequality,
\[\mu (\{X\in 2^\omega \altmid |f(X) - g(X)| \geq  \varepsilon\})\leq (\varepsilon\cdot\delta) / \varepsilon = \delta.\]
So outside a set of measure at most $\delta$, we have $|f(X) - g(X)|<\varepsilon$, and therefore
\[|f(X) - a|  <  r - 2\varepsilon \quad \Rightarrow \quad  |g(X) - a|  <  r - \varepsilon \quad \Rightarrow \quad  |f(X) - a|  <  r.\]
Expressing this with measures gives us
\begin{align*}
\mu (\{X\in 2^\omega \altmid |f(X) - a| <  r \}) 
&\geq \mu(\{X\in 2^\omega \altmid |g(X) - a|  <  r - \varepsilon \}) -\delta\\
&\geq \mu(\{X\in 2^\omega \altmid |f(X) - a|  <  r - 2\varepsilon \}) - 2\delta.
\end{align*}
Combining this with (\ref{eq:sup}) we get
\begin{align*}
\begin{split}
&\mu (\{X\in 2^\omega \altmid |f(X) - a| <  r \}) \\
&\qquad \qquad = \sup_{\varepsilon>0, \delta>0} \mu(\{X\in 2^\omega \altmid |g(X) - a|  <  r - \varepsilon \}) - \delta
\end{split}
\end{align*}
where $g$ depends on $\varepsilon$ and $\delta$.  Finally, recall that $\mu (\{X\in 2^\omega \altmid |g(X) - a|  <  r - \varepsilon\})$ is lower semicomputable from $a,r$.  Using this we can approximate $\mu (\{X\in 2^\omega \altmid |f(X) - a| <  r \})$ from below.
\end{proof}

The following lemma will be the key to this paper.

\begin{lemma}[Key lemma]\label{lem:key-lemma}
Let $t$ be a nonnegative lower semicomputable function with a computable integral $\int \! t \, d\mu$.  There is a uniformly computable sequence $\langle h_n\rangle$ of total computable functions $h_n\colon 2^\omega \rightarrow [0,\infty)$ such that $h_n \leq t$ everywhere and if $A$ is Schnorr random, there is some $n$ such that $h_n(A)=t(A)$.
\end{lemma}

\begin{proof}
Let $\langle g_k\rangle$ be a code for $t$, namely a sequence of total nonnegative computable functions such that $t = \sum_k g_k$.  Find a sequence $\langle f_n\rangle$ of partial sums $f_n = \sum_{k<k_n} g_k$ (where $\langle k_{n} \rangle$ is increasing) such that $\int \! (t - f_n) \, d\mu < 2^{-2n}$.  (This can be done since $\int \! g_n \, d\mu$ is uniformly computable from $n$.)  By Chebeshev's inequality, for any $c>0$,
\[\label{eq:Chebeshev} \mu (\{X\in 2^\omega \altmid t(X) - f_n(X) > c \}) \leq 2^{-2n}/c. \]
Moreover, we have this claim.

\begin{claim} There is a computable sequence $\langle c_n\rangle$ such that $2^{-n}<c_n<2^{-(n-1)}$ for each $n$ and $\mu (\{X\in 2^\omega \altmid t(X) - f_n(X) > c_n \})$ is uniformly computable from $n$.  
\end{claim}

\begin{proof}[Proof of claim.]
First, note that $t$ is $L^1$-computable. (Use the sequence $\langle f_n \rangle$ from earlier in the proof.)

Now, fix $n$.  Our goal is to find $c_n$.  Since, $t$ is $L^1$-computable, so is $t - f_n$.  Let $\nu$ to be the distribution of $t - f_n$.  By Proposition~\ref{prop:dist-comp}, $\nu$ is a computable distribution.  

Hence for any real $c$,
\[\mu (\{X\in 2^\omega \altmid t(X) - f_n(X) >  c \})= \nu ((c,\infty))\]
is lower semicomputable uniformly from $c$ and
\[\mu (\{X\in 2^\omega \altmid t(X) - f_n(X) \geq  c \})= \nu ([c,\infty))= 1 - \nu ((-\infty,c))\]
is upper semicomputable uniformly from $c$.

It is enough to find some $c$ in the desired interval such that 
\[\mu (\{X\in 2^\omega \altmid t(X) - f_n(X) =  c \}) = \nu(\{c\}) =0.\]
Indeed, the set of all $c$ such that $\nu(\{c\})=0$ is a computable intersection of dense $\Sigma^0_1$ sets.  (To see this, note that $\nu(\{c\})=1-\nu((-\infty,c)\cup(c,\infty))$ is upper semicomputable uniformly from $c$.  So $\{c\in \mathbb{R} \altmid \nu(\{c\}) < 2^{-n}\}$ is a $\Sigma^0_1$ set.  This set is also dense since there are at most countable many $c$ such that $\nu(\{c\})>0$.) Hence by the effective proof of the Baire category theorem (a basic diagonalization argument, see for example \citelow{Brattka2001}), one can effectively find $c_n$ in the desired interval such that $\mu (\{X\in 2^\omega \altmid t(X) - f_n(X) > c_n \})$ is computable.  This proves the claim.
\end{proof}

To define $h_m$, first set 
\[ h^m_n = \min\{f_n,f_{n-1}+c_{n-1}, \ldots , f_m + c_m\} \quad(\text{for } n > m) .\]
Define $h_m = \sup_{n > m} h^m_n$. For each $X$ and $n > m$ we have $|h_m(X) - h^m_n(X)| \leq c_n$.  So $h_m(X)$ is uniformly computable from $X$ and $m$.  Also, $h_m \leq t$ everywhere, and $t(X) > h_m(X)$ if and only if $t(X) > f_n(X) + c_n$ for some $n>m$.  Then
\[
\{X\in 2^\omega \altmid t(X) > h_m(X) \} = 
 \bigcup_{n>m} \underbrace{\{X\in 2^\omega \altmid t(X) - f_n (X) > c_n\}}_{=:V^m_n} =: U_m.
\]
Notice $V_n$ is $\Sigma^0_1$ uniformly in $n$. By the claim and by inequality~(\ref{eq:Chebeshev}), $\mu(V_n)$ is uniformly computable from $n$ and at most $2^{-n}$.  Therefore, $\mu(U_m)\leq \sum_{n>m} 2^{-n}=2^{-m}$ and $\mu(U_m)$ is uniformly computable from $m$.  Hence $\langle U_m\rangle$ is a Schnorr test.

Finally, for any Schnorr random $A$, there is some $m$ such the $A \notin U_m$. Hence $h_m(A) = t(A)$.  This completes the proof of the lemma.
\end{proof}

\begin{remark}\label{rem:proof-is-uniform}
Notice, this proof gives a uniform algorithm for converting the codes for $t$ and $\int \! t \, d\mu$ into a code for a Schnorr test $\langle U_m\rangle$ such that $t(X)$ is finite when $X \notin \bigcap_m U_m$.  However, the uniformity is not (and cannot be) independent of the codes.  Indeed even a different code for $\int \! t \, d\mu$ can change the sequences $\langle g_k\rangle$ and $\langle c_n\rangle$ in the proof leading to a different Schnorr test $\langle U_m\rangle$.
\end{remark}

As a corollary, we get another proof of Miyabe's characterization of Schnorr randomness via Schnorr integral tests. 

\begin{definition}\ 
\begin{enumerate}
  \item An \df{integral test} is a lower semicomputable function $t\colon 2^\omega \rightarrow [0,\infty]$ such that $ \int \! t \, d\mu <\infty$.
  \item (Miyabe \citelow{Miyabe:cr}) A \df{Schnorr integral test} is an integral test $t$ such that $ \int \! t \, d\mu$ is computable. 
\end{enumerate}
\end{definition}

\begin{proposition}[See for example \citelow{Downey:2010ve}]\ 
$X$ is Martin-L{\"o}f random if and only if there is no integral test $t$ such that $t(X)=\infty$.
\end{proposition}

\begin{corollary}\label{cor:int-test-comp}
If $t$ is a Schnorr integral test and $A$ is Schnorr random, then $t(A)$ is finite and computable from $A$.
\end{corollary}

\begin{proof}
By Lemma~\ref{lem:key-lemma}, there is some total computable $h_n$ such the $t(A)=h_n(A)$.
\end{proof}

\begin{corollary} [Miyabe \citelow{Miyabe:cr}]\label{cor:int-test}
$X$ is Schnorr random if and only if there is no Schnorr integral test $t$ such that $t(X)=\infty$.
\end{corollary}

\begin{proof}
($\Rightarrow$) Assume $A$ is Schnorr random and $t$ is a Schnorr test.  By Lemma~\ref{lem:key-lemma}, there is some total computable $h_n$ such the $t(A)=h_n(A)<\infty$.
($\Leftarrow$) Assume $A$ is not Schnorr random by the Schnorr test $\langle U_n\rangle$.  Then let $t = \sum_n \mathbf{1}_{U_n}$ where $\mathbf{1}_{U_n}$ is the characteristic function of $U_n$.
\end{proof}

\section{Van Lambalgen's theorem for Schnorr randomness}
\label{sec:vL-for-Schnorr}

As a warm up, we use the results from the previous section to give a simple proof of the ``$\Leftarrow$" direction of van Lambalgen's theorem for Schnorr randomness.  While our proof is simple, we argue that it is not a straightforward modification of van Lambalgen's theorem.  (Franklin and Stephan \citelow{Franklin:2011fr} also recently gave a very different proof of this result.)  

\begin{theorem} \label{thm:VL-Schnorr}
If $A$ is Schnorr  random and $B$ is Schnorr random relative to $A$ then $A \oplus B$ is Schnorr random.
\end{theorem}

\begin{proof}
Assume $A \oplus B$ is not Schnorr random and that $A$ is Schnorr random.  There is some Schnorr integral test $t$ such that  $t(A \oplus B)=\infty$.  We wish to show that $B$ is not Schnorr random relative to $A$.  For each $X$, define $t^X(Y) = t(X \oplus Y)$.  Clearly, $t^A(B)=t(A\oplus B)=\infty$.   

We will show that $t^A$ is a Schnorr integral test relative to $A$.  For each $X$, the function $t^X$ is lower semicomputable with a code uniformly computable from $X$. Define, $u(X)=\int \! t^X(Y) \, d\mu(Y)$.  Notice $u$ is lower semicomputable and by Fubini's theorem
\[ \int \! u(X) \, d\mu(X)=\iint \! t(X \oplus Y) \, d\mu(Y) \, d\mu(X) = \int \! t \, d\mu\]
which is computable.
It follows that $u$ is a Schnorr integral test.  \emph{Since $A$ is Schnorr random, by Corollary~\ref{cor:int-test-comp}, $u(A)$ is computable from $A$, and therefore so is $\int \! t^A(Y) \, d\mu(Y)=u(A)$.} Hence $t^A$ is a Schnorr integral test relative to $A$, and $B$ is not Schnorr random relative to $A$.
\end{proof}

\begin{remark}
Notice the emphasized line in the above proof.  The key difficulty in adapting the standard (integral test) proof of van Lambalgen's theorem to Schnorr randomness is that for Martin-L{\"o}f randomness one need only show that $\int \! t^A(Y) \, d\mu(Y)$ is \emph{finite} while here we must show it is \emph{computable from $A$}.  The same difficulties exist when trying to adapt a proof using Martin-L{\"o}f tests; given a Schnorr test $\langle U_n\rangle$ one must show the measure of $U^A_n := \{Y\in 2^\omega \altmid A \oplus Y \in U_n \}$ is computable from $A$.

Consider this latter case.  One may incorrectly think that $\mu(U^A_n)$ is uniformly computable from $A$, any code for $U_n$, and the measure $\mu(U_n)$.%
\footnote{%
Indeed, this was the error in \citelow{Miyabe:2011zr} when proving van Lambalgen's theorem for truth-table Schnorr randomness.%
}  This is false, as the following picture shows. The two open sets depicted (on $2^\omega \times 2^\omega$) are almost the same except one contains a small gap.  Such a small gap could significantly change the value of $\mu(U^A_n)$.  It is impossible to determine in a fixed number of steps whether such a small gap exists.

\begin{center}
\includegraphics{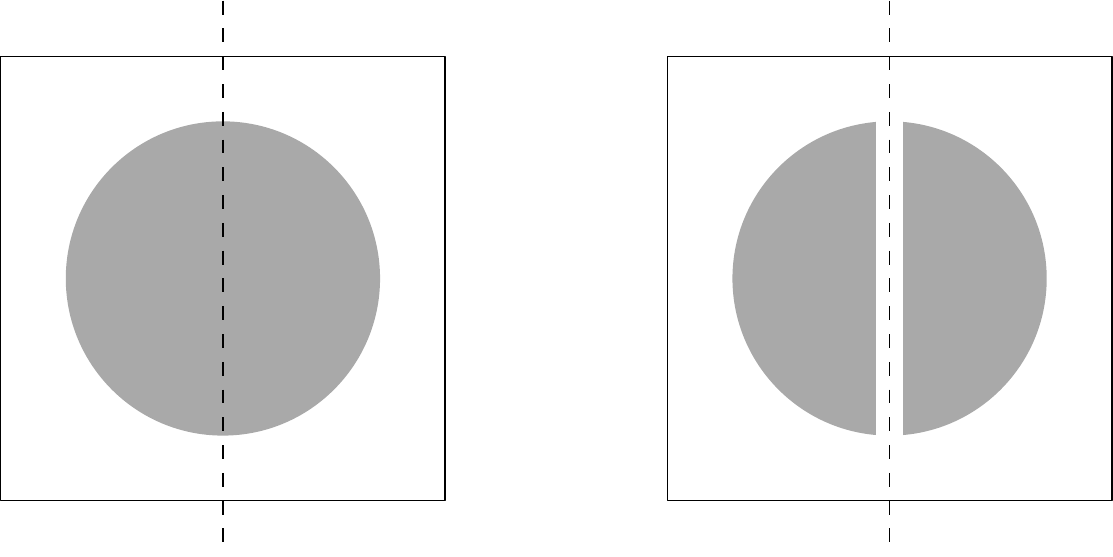}
\end{center}

However, the key lemma does show that $\mu(U^A_n)$  is computable from $A$ if $A$ is Schnorr random.  One could say such a function is ``nearly computable," i.e.\  it is computable outside an arbitrary small set.  This is an effectivization of Littlewood's ``three principles" of measure theory, particularly Lusin's theorem which says that a measurable function is ``nearly continuous".  It is also the basis behind the layerwise-computability framework of Hoyrup and Rojas \citelow{Hoyrup:2009ly} and its extensions to Schnorr randomness.  (See also \citelow{Miyabe:cr,Pathak:2012bh, Rute:dq}.)
\end{remark}

\section{Uniformly relative Schnorr randomness}
\label{sec:tt-Schnorr}

In this section we give the correction to Miyabe's proof of van Lambalgen's theorem for uniformly relative Schnorr randomness.  But first, we define uniformly relative Schnorr randomness.

Recall, that a Schnorr test can be encoded by a function $f \in \mathbb{N}^\mathbb{N}$ which encodes a listing of basic open sets for each $n$ which union to $U_n$ (as in Definition~\ref{def:open}) and also encodes a fast Cauchy sequence of rationals converging to each measure $\mu(U_n)$.  (Recall, a sequence of rationals $\langle q_n\rangle$ is \df{fast-Cauchy} if for all $n \geq m$, we have $|q_n - q_m| \leq 2^{-m}$.)
 
\begin{definition}\label{def:Uniform-Schnorr}
A \df{uniform Schnorr test} is a total computable map $\Phi\colon 2^\omega \rightarrow \mathbb{N}^\mathbb{N}$ such that each $\Phi(X)$ encodes a Schnorr test.  Also, call a collection $\langle U^X_n \rangle_{n\in\mathbb{N}, X\in 2^\omega}$ a \df{uniform Schnorr test} if it is given by a map $\Phi\colon 2^\omega \rightarrow \mathbb{N}^\mathbb{N}$ as above.
\end{definition}

\begin{definition}[Miyabe \citelow{Miyabe:2011zr}, following Franklin and Stephan \citelow{Franklin:2010ys}]
Let $A,B \in 2^\omega$.  Say $A$ is \df{Schnorr random uniformly relative to $B$} if there is no \df{uniform Schnorr test}  $\langle U^X_n \rangle_{n\in\mathbb{N}, X\in 2^\omega}$ such that $A\in \bigcap_n U^B_n$. 
\end{definition}

\begin{remark}
Uniformly relative Schnorr randomness has previously been called truth-table Schnorr randomness.  This new name better reflects the exact nature of the tests.  See Section~\ref{sec:equiv} for more discussion on the differences between the two approaches to relativizing Schnorr randomness.
\end{remark}

Also, it is possible to define uniform tests of other types similarly.

\begin{definition}
A \df{uniform Schnorr integral test} is a total computable map $\Phi\colon 2^\omega \rightarrow \mathbb{N}^\mathbb{N}$ such that each $\Phi(X)$ encodes a Schnorr integral test.  Also, call a collection $\langle t^X \rangle_{X\in 2^\omega}$ a \df{uniform Schnorr integral test} if it is given by a map $\Phi\colon 2^\omega \rightarrow \mathbb{N}^\mathbb{N}$ as above.
\end{definition}

\begin{proposition} \label{prop:unif-Sch-int-test}
Let $A,B \in 2^\omega$.  Then $A$ is Schnorr random uniformly relative to $B$ if and only if there is no uniform Schnorr integral test  $\langle t^X \rangle_{X\in 2^\omega}$ such that $t^B(A) = \infty$. 
\end{proposition}

\begin{proof}
Follow the proof of Corollary~\ref{cor:int-test}.  The proof is uniform in that given a code for a test, some code for the other type of test is computable uniformly from the first code (see Remark~\ref{rem:proof-is-uniform}).
\end{proof}

\begin{remark}
While the reader should not confuse our definition of uniform test with that of Levin \citelow{Levin:1976uq}, the two ideas are closely related.  The main difference is that in Levin's definition the uniform test is indexed by probability measures $\xi$ and the test $t^\xi$ was with respect to the measure $\xi$ (i.e.\ $\int\! t^\xi\, d\xi < \infty$).
\end{remark}

Now, we prove the main result.  The proof will closely follow that of Theorem~\ref{thm:VL-Schnorr}, but will use the full power of the key lemma (Lemma~\ref{lem:key-lemma}).

\begin{theorem} \label{thm:VL-tt-Schnorr}
$A \oplus B$ is  Schnorr random if and only if $A$ is Schnorr random and $B$ is Schnorr random uniformly relative to $A$.
\end{theorem}

\begin{proof}
($\Rightarrow$) This direction is correctly proved in \citelow{Miyabe:2011zr}.

($\Leftarrow$) Assume $A \oplus B$ is not Schnorr random and that $A$ is Schnorr random.  There is some Schnorr integral test $t$ such that  $t(A \oplus B)=\infty$.  We wish to show that $B$ is not Schnorr random uniformly relative to $A$.  For each $X$, define $t^X(Y) = t(X \oplus Y)$.  Clearly, $t^A(B)=t(A \oplus B)=\infty$.

While $\langle t^X\rangle_{X\in 2^\omega}$ may not be a uniform Schnorr integral test, we will construct a uniform Schnorr integral test $\langle \hat{t}^X\rangle_{X\in 2^\omega}$ such that $\hat{t}^A = t^A$.  For each $X$, the function $t^X$ is lower semicomputable with a code uniformly computable from $X$. Define, $u(X)=\int \! t^X(Y) \, d\mu(Y)$.  Notice $u$ is lower semicomputable  and by Fubini's theorem
\[ \int \! u(X) \, d\mu(X)=\iint \! t(X \oplus Y) \, d\mu(Y) \, d\mu(X) = \int \! t \, d\mu\]
which is computable.
It follows that $u$ is a Schnorr integral test.  By the key lemma (Lemma~\ref{lem:key-lemma}), there is some total computable function $h\leq u$ such that $u(A)=h(A)$. Let $\hat{t}^X$ be $t^X$, but enumerated only so that $\int \! \hat{t}^X \, d\mu = h(X)$.  More formally, for all $X$ let $\langle g^X_n\rangle$ be some code for $t$, namely some sequence of total nonnegative computable functions such that $t^X = \sum_n g^X_n$.  Label the integrals of the partial sums $c^X_n = \int \! \sum_{k < n} g^X_k \, d\mu$. Define 
\[ \hat{t}^X = \sum_n
  \begin{cases}
  g^X_n & \text{if } c^X_{n+1}  < h(X) \\
  \left(\frac{h(X)-c^X_{n}}{c^X_{n+1}- c^X_{n}} \right)\cdot g^X_{n} & \text{if } c^X_n \leq h(X) \leq c^X_{n+1} \\
  0 & \text{if } h(X) < c^X_n \\
  \end{cases}.
\]
Then $\langle \hat{t}^X\rangle_{X\in 2^\omega}$ is a uniform Schnorr integral test where $\int \! \hat{t}^X \, d\mu = h(X)$ and if $h(X)=u(X)$ then $\hat{t}^X=t^X$.  Since  $h(A)=u(A)$, we have $\hat{t}^A(B)=t^A(B)=\infty$.  Therefore $B$ is not Schnorr random uniformly relative to $A$.
\end{proof}

\section{Uniformly relative computable randomness}
\label{sec:comp-rand}

In this section we give a weaker version of van Lambalgen's theorem for uniformly relative computable randomness.

Recall that a \df{martingale} is a function $d\colon 2^{<\omega} \rightarrow \mathbb[0,\infty)$ such that $d(\sigma0)+d(\sigma 1) = 2d(\sigma)$ for all $\sigma \in 2^{<\omega}$. We will say a martingale $d$ \df{succeeds} on a set $A\in2^\omega$ if $\limsup_n d(A\upharpoonright n) = \infty$.  A set $A$ is \df{computably random} (or \df{recursively random}) if there is no computable martingale $d$ which succeeds on $A$.  A well-known alternate characterization is that $A$ is computably random if and only if there is no computable martingale $d$ such that $\lim_n d(A\upharpoonright n) = \infty$.

\begin{definition}
A \df{uniform martingale test} is a total computable map $\Phi\colon 2^\omega \rightarrow \mathbb{N}^\mathbb{N}$ such that each $\Phi(X)$ encodes a martingale.  Also, call a collection $\langle d^X\rangle_{X\in 2^\omega}$ a \df{uniform martingale test} if it is given by a map $\Phi\colon 2^\omega \rightarrow \mathbb{N}^\mathbb{N}$ as above.
\end{definition}

\begin{definition}[Miyabe \citelow{Miyabe:2011zr}]
$A$ is \df{computably random uniformly relative to $B$} if there is no uniform martingale test $\langle d^X\rangle_{X\in 2^\omega}$ such that $d^B$ succeeds on $A$.
\end{definition}

Uniformly relative computable randomness was previously called ``truth-table reducible randomness" in \citelow{Miyabe:2011zr}.  In Section~\ref{sec:equiv} we give an alternative definition using truth-table reducibility in the spirit of Franklin and Stephan.

The first author \citelow{Miyabe:2011zr} claimed the following, but the proof of the ``$\Leftarrow$" direction is incorrect.%
\footnote{%
The logical error in \citelow[Theorem~5.10]{Miyabe:2011zr} is in the following formula.
\[ \mu(\{Y\mid A\oplus Y\in W_{n}\}\cap[\tau])\leq\frac{\nu(A\upharpoonright m\oplus\tau)}{\nu(A\upharpoonright m\oplus\lambda)}2^{-n} \]
The first $A$ should be $[A\upharpoonright m]$ and the $\in$ should be $\subseteq$. Then the next line about letting $m=0$ does not hold.

The conceptual error is as follows.  The proof as usual starts by assuming $A \oplus B$ is not computably random.  The witnessing test is a bounded Martin-L{\"o}f test with bounding measure $\nu$.  The proof attempts to show that the bounding measure $\tau \mapsto \nu(\lambda \oplus \tau)$ (where $\lambda$ is the empty string) witnesses that $B$ is truth-table reducibly random relative to $A$.  However, $\tau \mapsto \nu(\lambda \oplus \tau)$ (that is, the marginal distribution of the second coordinate) is computable (not just computable relative to $A$). Hence this argument would show that $B$ is not computably random, which is too strong.%
}
\begin{quote}
$A \oplus B$ is computably random if and only if $A$ is computably random and $B$ is computably random uniformly relative to $A$.
\end{quote}

However, we can prove this weaker version of van Lambalgen's theorem for uniformly relative computable randomness.

\begin{theorem}\label{thm:VL-tt-cr}
$A \oplus B$ is computably random if and only if $A$ is computably random uniformly relative to $B$ and $B$ is computably random uniformly relative to $A$.
\end{theorem}

Before giving the formal proof, we give the main idea.  Assume some computable martingale $d$ satisfies $\lim_n d((A \oplus B)\upharpoonright n) = \infty$.  Then split $d$ into two martingales $d^{B}_0$ and $d^{A}_1$. Here $d^{B}_0$ uses $B$ to bet on the $n$th bit of some set $X$ with the same relative amount that $d$ bets on the $2n$th bit of $X \oplus B$.  The martingale $d^{A}_1$ is defined similarly.  We will show $d((A + B) \upharpoonright 2n) = d^{B}_0(A\upharpoonright n) \cdot d^{A}_1(B\upharpoonright n)$.  So  either $d^{B}_0$ succeeds on $A$ or $d^{A}_1$ succeeds on $B$.

To prove the theorem, we introduce the following notation.  Given a martingale $d$ such that $d(\sigma)\neq 0$ for all $\sigma \in 2^{<\omega}$, define $\widetilde{d}(\sigma i) = \frac{d(\sigma i)}{d(\sigma)}$ for $i\in \{0,1\}$.  This codes the martingale by the relative changed at each step.  The martingale can be recovered by $d(\sigma) = \prod_{m=1}^{|\sigma|} \widetilde{d}(\sigma \upharpoonright m)$ (assuming $d(\varnothing)=1$).   Also $\widetilde{d}$ codes a martingale if and only if $\widetilde{d}(\sigma0) + \widetilde{d}(\sigma1) = 2$ for all $\sigma$.  Call $\widetilde{d}$ the \df{multiplicative representation} of the martingale $d$.

Another notion will be as follows.  If $\sigma,\tau\in 2^{<\omega}$, then define $f=\sigma \oplus \tau$ as a partial function from $\mathbb{N}$ to $\{0,1\}$ such that $f(2n)=\sigma(n)$ if $n<|\sigma|$ and $f(2n+1) = \tau(n)$ if $n<|\tau|$.

\begin{proof}[Proof of Theorem~\ref{thm:VL-tt-cr}]
Assume $A \oplus B$ is not computably random.  Then there is a computable martingale $d$ such that $\lim_n d((A \oplus B)\upharpoonright n) = \infty$.  (Without loss of generality, $d(\varnothing)=1$ and $d(\sigma)>0$ for all $\sigma$.) Define two uniform martingale tests $\langle d^X_0 \rangle_{X\in 2^\omega}$ and $\langle d^X_1 \rangle_{X\in 2^\omega}$ by  their corresponding multiplicative representations as follows.
\begin{align*}
  \widetilde{d}^Y_0(\sigma) &=  \widetilde{d}(\sigma \oplus (Y\upharpoonright |\sigma| - 1))\\
  \widetilde{d}^X_1(\tau) &= \widetilde{d}((X\upharpoonright |\tau|)  \oplus \tau)
\end{align*}
Then $\widetilde{d}^Y_0$ (and similarly $\widetilde{d}^X_1$) is a multiplicative martingale since 
\[ 
\widetilde{d}^Y_0(\sigma0) + \widetilde{d}^Y_0(\sigma1) = \widetilde{d}((\sigma \oplus (Y\upharpoonright |\sigma|))^{\smallfrown}0) + \widetilde{d}((\sigma \oplus (Y\upharpoonright |\sigma|))^{\smallfrown}1)=2.
\]
Now for $X=A$ and $Y=B$ we have that
\begin{align*}
d^{B}_0(A \upharpoonright n) \cdot d^{A}_1(B \upharpoonright n) &= \left( \prod_{m=1}^{n} \widetilde{d}^{B}_0(A \upharpoonright m) \right) \cdot \left( \prod_{m=1}^{n} \widetilde{d}^{A}_1(B \upharpoonright m) \right) \\
&= \prod_{m=1}^{2n} \widetilde{d}((A + B) \upharpoonright m)\\
&= d((A + B) \upharpoonright 2n) \rightarrow \infty
\end{align*}
Then either $d^{B}_0$ succeeds on $A$ or $d^{A}_1$ succeeds on $B$.  It follows that one of $A,B$ is not computably random uniformly relative to the other.
\end{proof}

We also obtain this easy corollary.  Since it follows from the ``$\Rightarrow$" direction of Theorem~\ref{thm:VL-tt-cr}, it was provable from the results of \citelow{Miyabe:2011zr}.

\begin{corollary}
There exist sets $A$ and $B$ such the $A$ is computably random \emph{uniformly relative} to $B$, but $A$ is not computably random \emph{relative} to $B$.
\end{corollary}

\begin{proof}
Merkle et al.\ \citelow{Merkle:2006vn} showed that there is are two computable randoms $A,B$ such that $A \oplus B$ is computably random, but $A$ is not computably random relative to $B$.  However, by Theorem~\ref{thm:VL-tt-cr}, $A$ is computably random uniformly relative to $B$.
\end{proof}

\subsection{Remark on Kolmogorov-Loveland and Martin-L{\"o}f randomness}\label{sub:rmk-on-KL-ML}
\label{sec:tt-KL}

Recall that Kolmogorov-Loveland randomness is a notion of randomness similar to computable randomness, except the martingales do not need to bet on the bits in order.  These are called \df{nonmonotonic martingales}.  Also, the martingales may be \df{partial}, i.e.\ not defined on all inputs.  (See \citelow{Merkle:2006vn, Downey:2010ve, Nies:2009qf} for formal definitions and more information on Kolmogorov-Loveland randomness.) The proof of Theorem~\ref{thm:VL-tt-cr} is similar to the proof of the following.

\begin{theorem}[Merkle et al.\ \citelow{Merkle:2006vn}]\label{thm:VL-KL}
$A \oplus B$ is Kolmogorov-Loveland random if and only if $A$ and $B$ are Kolmogorov-Loveland random relative to each other.
\end{theorem}

One may ask if there is a notion of ``uniformly relative Kolmogorov-Loveland randomness" using uniform partial nonmonitonic martingale tests.  The answer is that being Kolmogorov-Loveland random uniformly relative to $B$ is the same as being Kolmogorov-Loveland random relative to $B$.  Any nonmonotonic martingale computable from $B$, is easily extended to a uniform partial nonmonotonic martingale.  This is easy to do because the nonmonotonic martingale needs only be partial computable from the oracle.

Similarly, ``Martin-L{\"o}f random uniformly relative to $B$" is equivalent to Martin-L{\"o}f random relative to $B$.  If $\langle U_n \rangle$ is a Martin-L{\"o}f test relative to $B$, then there is a uniform Martin-L{\"o}f test $\langle U^X_n \rangle_{n\in \mathbb{N}, X\in 2^\omega}$ given by using the same algorithm (as for $\langle U_n \rangle$) to enumerate the basic open sets of $U^X_n$, but we stop enumerating the basic open sets if doing so will make $\mu(U^X_n)>2^{-n}$.

\section{Truth-table Schnorr and truth-table reducible randomness}
\label{sec:equiv}

\subsection{Truth-table Schnorr randomness is equivalent to uniformly relative Schnorr randomness}
Truth-table Schnorr randomness was first defined by Franklin and Stephan \citelow{Franklin:2010ys} as follows using a truth-table relativized martingale test.  Recall that a function $f\in \mathbb{N}^\mathbb{N}$ is \df{truth-table reducible} to $A\in2^\omega$ if there is a total computable functional $\Phi\colon 2^\omega \rightarrow \mathbb{N}^\mathbb{N}$ such that $f=\Phi(A)$.

\begin{definition}[Franklin and Stephan \citelow{Franklin:2010ys}]\label{defn:tt-Schnorr}
A set $A$ is \df{truth-table Schnorr relative to $B$} if there is no pair $(d,f)$ consisting of a martingale $d$ with code truth-table reducible to $B$ and a function $f\colon \mathbb{N} \rightarrow \mathbb{N}$ truth-table reducible to $B$ such that $\exists^\infty n\, [d(A \upharpoonright f(n)) \geq n]$.
\end{definition}

Franklin and Stephan also remark that one may take $f$ in the above definition to be computable (instead of truth-table reducible to $B$) with no loss.

The first author showed in \citelow{Miyabe:2011zr} that truth-table Schnorr randomness is equivalent to (what we call) uniformly relative Schnorr randomness, but there was a small gap in the proof which we feel would be instructive to fill in here.

\begin{lemma}\label{lem:tt-to-uniform}
Let $d$ be a martingale (with some code) truth-table reducible to $A$ and let $f\colon \mathbb{N} \rightarrow \mathbb{N}$ be truth-table reducible to $A$.  Then there is a uniform martingale test $\langle \hat{d}^X\rangle_{X\in2^\omega}$ and a uniform function $\langle \hat{f}^X\rangle_{X\in2^\omega}$ such that $\hat{d}^A=d$ and $\hat{f}^A=f$.
\end{lemma}

\begin{proof}
Let $d$ be a martingale truth-table reducible to $A$.  And let $\Phi\colon 2^\omega \times 2^{<\omega} \rightarrow \mathbb{N}^\mathbb{N}$ be a total computable functional such that $\Phi(A,\sigma)$ encodes the fast-Cauchy code for $d(\sigma)$.  Define $d^X(\sigma)$ to be the real coded by $\Phi(X,\sigma)$.  Note, that there is no guarantee that, first, $\Phi(X,\sigma)$ is a Cauchy code for a nonnegative real for every $X$ and $\sigma$, and that, second, $d^X$ is a martingale for every $X$.

The first issue is easily fixed.   We use a folklore trick to force $\Phi(X,\sigma)$ to be a fast Cauchy code for a nonnegative number.  Let $\langle q_0, q_1, \ldots\rangle$ be the sequence of rationals given by $\Phi(X,\sigma)$.  Firstly, replace each $q_n$ with $\max\{q_n,0\}$.  Secondly, find the first $n$, if any, such that $|q_n - q_m| > 2^{-m}$ for $m \leq n$, then change the code to be $\langle q_0,q_1,\ldots, q_{n-1} \rangle^{\smallfrown}\langle q_{n-1},q_{n-1}\ldots\rangle$.  Notice this does not change the value of $d(A)$.

The second issue is also easily handled.  Assuming, now that each $d^X(\sigma)$ is a nonnegative real, define $\hat{d}^X(\sigma)$ by recursion as follows.
\begin{align*}
  \hat{d}^X(\varnothing) &= d^X(\varnothing) \\
  \hat{d}^X(\sigma 0) &= \min\{d^X(\sigma 0), 2 \hat{d}^X(\sigma)\} \\
  \hat{d}^X(\sigma 1) &= 2\hat{d}^X(\sigma) - \hat{d}^X(\sigma 0)
\end{align*}
It is easy to check $\hat{d}^X$ is a nonnegative martingale whose code is uniformly computable from the code for $d^X$.  Also if  $d^X$ is already a martingale, then $\hat{d}^X = d^X$.  In particular, $\hat{d}^A=d^A$.

Last if $f$ is truth-table reducible to $A$, then there is a total functional $\Psi\colon 2^\omega \rightarrow \mathbb{N}^\mathbb{N}$ such that $\Psi(A)=f$. Define $\hat{f}^X=\Psi(X)$.
 \end{proof}

Now we can show that the definitions are equivalent.

\begin{proposition}\label{prop:defs-same}
A set $A$ is Schnorr uniformly relative to $B$ if and only if $A$ is truth-table Schnorr relative to $B$.
\end{proposition}

\begin{proof}
By Lemma~\ref{lem:tt-to-uniform} (and its trivial converse) truth-table Schnorr randomness is equivalent to that obtained from uniform martingale tests of the above type. Now similar to Proposition~\ref{prop:unif-Sch-int-test}, it is enough to show a uniform martingale test (of the above type) can effectively be converted to a Schnorr test which covers the same points that the martingale succeeds on, and vice versa. Indeed, the proofs in the literature are effective in this regard (see \citelow{Franklin:2010ys,Downey:2010ve,Nies:2009qf}).
\end{proof}

\begin{remark}
There is an important subtlety in the last proof similar to Remark~\ref{rem:proof-is-uniform}.  We showed it is possible to compute \emph{a code for} one test uniformly from \emph{a code for} another test. However, it is not necessarily possible to do so in a way that is independent of the choice of codes.  For example, it is known that one may effectively replace a real-valued martingale $d$ with a rational-valued martingale $\hat{d}$ that succeeds on the same points.  However, there is some $d$ such that different codes for $d$ lead to different rational approximations.  This will become an issue if instead of relativizing with respect to a set $X\in2^\omega$, one relativized with respect to a real $x\in[0,1]$.
\end{remark}

\subsection{Truth-table reducible randomness is equivalent to uniformly relative computable randomness}

We also have a similar result for computable randomness.  For this reason, uniformly relative computable randomness is also known as \df{truth-table reducible randomness} \citelow{Miyabe:2011zr}.

\begin{proposition}\label{prop:defs-same-cr}
A set $A$ is computably random uniformly relative to $B$ if and only if for all martingales $d$ with a code truth-table reducible to $B$, we have that $\limsup_n d(A\upharpoonright n) < \infty$.
\end{proposition}

\begin{proof}
By the proof of Lemma~\ref{lem:tt-to-uniform} it is possible to pass from a martingale truth-table reducible to $B$ to a uniform martingale test.  The converse is trivial.
\end{proof}

\subsection{Characterizations of ``truth-table Schnorr randomness" by other tests}

Notice that the motivation behind Definition~\ref{defn:tt-Schnorr} is to say that $A$ is ``truth-table Schnorr random" relative to $B$ if there is no test for Schnorr randomness (e.g.\ Schnorr test, integral test, martingale test, etc.)\ which is truth-table reducible to $B$ and witnesses that $A$ is not random for that test.  Unfortunately, this method is very sensitive to the choice of test.  Lemma~\ref{lem:tt-to-uniform} does not hold for most other characterizations of Schnorr randomness, including the usual martingale and Schnorr test characterizations.  More specifically we will show that the following two natural-looking definitions of ``truth-table Schnorr randomness" are in fact strictly stronger than uniformly relative Schnorr randomness.

\begin{definition}\label{def:bad-truth-table}\ 
\begin{enumerate}
\item A set $A$ is \df{``truth-table (martingale) Schnorr random" relative to $B$} if there is no pair $(d,f)$ consisting of a martingale $d$ which a code truth-table reducible to $B$ and an order function $f\colon \mathbb{N} \rightarrow \mathbb{N}$ truth-table reducible to $B$ such that $\exists^\infty n\, [d(A \upharpoonright n) \geq f(n)]$.  (Recall, an \df{order} is an unbounded increasing function.)
\item The set $A$ is \df{``truth-table (test) Schnorr random" relative to $B$} if there is no sequence $\langle U_n \rangle$ of open sets with codes uniformly truth-table reducible to $B$ such that $\mu(U_n)=2^{-n}$ for all $n$ and $A\in U_n$ for all $n$. 
\end{enumerate}
\end{definition}

\begin{proposition}[Due to Stephan and Franklin  \citelow{Franklin:2010ys}]\label{prop:bad-mart-test}
There exist $A$ and $B$ such that $A$ is Schnorr random uniformly relative to $B$ but not ``truth-table (martingale) Schnorr random" relative to $B$.
\end{proposition}

\begin{proof}Downey, Griffiths, and LaForte \citelow{Downey2004a} showed that there is a Turing complete Schnorr trivial $B$.  Nies, Stephan, and Terwijn \citelow{Nies:2005nx} showed there is a Schnorr random $A \equiv_T \emptyset'$ that is not computably random. Take this to be our $A$ and $B$.  Franklin and Stephan \citelow{Franklin:2010ys} showed that every Schnorr trivial is low for truth-table Schnorr random.  This means that since $A$ is Schnorr random, we have that $A$ is truth-table Schnorr random relative to $B$, and hence Schnorr random uniformly relative to $B$.

Given such $A,B$ as above, Franklin and Stephan \citelow[Theorem~2.2]{Franklin:2010ys} construct a computable martingale $d$ (hence truth-table reducible to $B$) and an order function $f$ truth-table reducible to $B$ such that $\exists^\infty n\, [d(A \upharpoonright n) \geq f(n)]$.  Therefore $A$ is not ``truth-table (martingale) Schnorr random" relative to $B$.
\end{proof}

\begin{proposition}\label{prop:bad-tt-test}
``Truth-table (test) Schnorr randomness" is equivalent to relative Schnorr randomness.  Hence, there exist $A$ and $B$ such that $A$ is Schnorr random uniformly relative to $B$ but not ``truth-table (test) Schnorr random" relative to $B$.
\end{proposition}

\begin{proof}
We will show that ``truth-table (test) Schnorr randomness" implies relative Schnorr randomness.  (The other direction is trivial.) Take a Schnorr test $\langle U_n \rangle$ computable relative to $B$.  We may assume $\mu(U_n)=2^{-n}$.  It remains to show that a code for $\langle U_n \rangle$ is truth-table reducible to $B$.  There is a partial computable function $\Phi\colon 2^\omega \times \mathbb{N} \times \mathbb{N} \rightarrow  \mathbb{N}$  such that $\Phi(B,-,-)$ encodes $\langle U_n \rangle$.  Namely, each
$\Phi(B,n,m)$ encodes a basic open set $C^B_{n,m}$ where $U_n=\bigcup_m C^B_{n,m}$.  (Recall, we allow $\varnothing\subset2^\omega$ as a basic open set.)  Then one can modify $\Phi$ to be total by repeatedly adding $\varnothing$ to the code.  Namely, define $\Psi\colon 2^\omega \times \mathbb{N} \times \mathbb{N} \rightarrow  \mathbb{N}$ as follows. If $m$ encodes the pair $\langle m',s \rangle$, and if $\Phi(X,n,m')$ halts by stage $s$, then let $\Psi(X,n,m)=\Phi(X,n,m')$.  Otherwise let $\Psi(X,n,m)$ encode $\varnothing$.  Since $\Psi$ is total computable and $\Psi(B,-,-)$ still encodes $\langle U_n \rangle$, we have that $\langle U_n \rangle$ has a code  truth-table reducible to $B$ as desired.

The rest of the proof is similar to the previous one.  Let $B$ be a Turing complete Schnorr trivial set.  Let $A \equiv_T \emptyset'$ be some Schnorr random. Hence $A$ is Schnorr random uniformly relative to $B$.  However, since $B \geq_T A$, then $A$ is not Schnorr random relative to $B$, and hence not ``truth-table (test) Schnorr random" relative to $B$.
\end{proof}

\begin{remark}A little thought reveals what is missing in Definition~\ref{def:bad-truth-table}. It is not in general possible to extend an order $f$ truth-table reducible to $B$ to a uniform order $\langle f^X\rangle_{X\in2^\omega}$.  To do this, one would also need that the rate of growth of $f$ is truth-table reducible to $B$.  A similar phenomenon happens with Schnorr tests.  One needs not only that the measure of $\mu(U_n)$ is truth-table reducible to $B$, but that given some code $\langle C_s \rangle$ for each $U_n$ which is truth-table reducible to $B$, the rate of convergence of $\mu\left(\bigcup_{s<n} C_s \right)$ must also be truth-table reducible to $B$.  After making these changes, then both parts of Definition~\ref{def:bad-truth-table} would be equivalent to uniformly relative Schnorr randomness.
\end{remark}

\section*{Acknowledgments}
We would like to thank Rod Downey and Denis Hirschfelt for encouraging us to publish this finding, as well as Johanna Franklin and Frank Stephan for clarifying the significance of Theorem~2.2 in \citelow{Franklin:2010ys} which led to the proof of Proposition~\ref{prop:bad-mart-test}.  Last, we would like to thank the anonymous referee for many helpful corrections and suggestions.  

The first author was supported by GCOE, Kyoto University and JSPS KAKENHI (23740072).

\bibliographystyle{plain}
\bibliography{paper_van_Lam_Schnorr}

\end{document}